\documentclass[10pt]{article}

\usepackage[a4paper,margin=30mm]{geometry}
\usepackage{amsmath,amssymb,amsthm,mathtools}
\usepackage{enumitem}
\usepackage{microtype}
\usepackage[hidelinks]{hyperref}
\usepackage{tikz}
\newtheorem{theorem}{Theorem}[section]
\newtheorem{proposition}[theorem]{Proposition}
\newtheorem{lemma}[theorem]{Lemma}

\theoremstyle{definition}
\newtheorem{example}[theorem]{Example}
\newtheorem{definition}[theorem]{Definition}
\theoremstyle{remark}

\newcommand{\G}{\mathcal G}
\newcommand{\V}{\mathcal V}
\newcommand{\E}{\mathcal E}
\newcommand{\RR}{\mathbb R}
\newcommand{\Tor}{\mathcal T}
\newcommand{\ip}[2]{\left\langle #1,#2\right\rangle}

\title{Sharp Bounds on Torsional Rigidity of Metric Graphs
	with Signed \texorpdfstring{\(\delta\)}{delta}-Vertex Strengths}
\author{Sedef \"Ozcan}
\date{}

\begin{document}
	\maketitle
	
	\begin{abstract}
		We study torsional rigidity on finite compact connected metric graphs
		with real \(\delta\)-vertex strengths in the regime of strictly positive
		torsion. For a fixed total length \(L>0\), let \(I_+\) be the sum of
		the positive strengths and \(I_-\) the sum of the magnitudes of the
		negative strengths, with \(I_+>I_-\). We determine the sharp infimum
		and the exact threshold for a finite universal supremum. The supremum
		is finite precisely when \(I_-<I_+/(1+LI_+)\). In this regime, an
		interval with strengths \(I_+\) and \(-I_-\) at opposite endpoints
		attains the sharp Saint--Venant bound. At and beyond the threshold,
		torsional rigidity is unbounded even among three-vertex paths with
		the same data and strictly positive torsion. The sharp infimum is
		\(L^2/(I_+-I_-)\) throughout the range \(I_+>I_-\ge0\). It is not
		attained by any finite graph and is approached by graphs with two
		vertices joined by increasingly many parallel edges of equal length,
		preserving both strength totals separately. A vertex-matrix formulation
		and explicit path examples clarify the relation between strength
		placement, spectral positivity, and solvability.
	\end{abstract}
	
	
	\section{Introduction}\label{sec:introduction}
Metric graphs provide a framework for differential equations on networks:
each edge is an interval with a prescribed length, and vertex conditions
couple the solutions on adjoining edges. Their geometry includes both
the edge lengths and the way the edges are connected; see
\cite{BerkolaikoKuchment2013,Mugnolo2019} for general accounts of this
framework. In this setting, spectral and geometric estimates must also
account for the conditions at the vertices.

In this paper, we consider finite compact connected metric graphs
\(\G=(\V,\E)\) with \(\delta\)-vertex conditions.
These conditions require the edgewise solutions to meet continuously
at the vertices and prescribe a balance of derivatives governed by
real parameters \(\alpha_{\mathrm v}\), called vertex strengths.
We collect the strengths in the vector
\(\alpha=(\alpha_{\mathrm v})_{\mathrm v\in\V}\) and denote the
associated Laplacian by \(-\Delta_{\G,\alpha}\).
Our setting allows both positive and negative strengths.

Torsional rigidity, a quantity originating in elasticity, measures the
total response to constant forcing. For a bounded domain
\(\Omega\subset\RR^n\), the Dirichlet torsional rigidity admits the
variational characterization \cite{Brasco2014}
\[
\Tor(\Omega)
=\max_{0\ne u\in H_0^1(\Omega)}
\frac{\left(\int_\Omega |u(x)|\,dx\right)^2}
{\int_\Omega |\nabla u(x)|^2\,dx}.
\]
The maximum is attained by the torsion function, namely the weak
solution of \(-\Delta u=1\) in \(\Omega\) with zero boundary values.
This formulation connects torsional rigidity to isoperimetric
inequalities and principal eigenvalues. The classical Saint--Venant
inequality identifies the disc as the maximizer among planar domains
of a given area \cite{Polya1948}.

On a metric graph, when \(-\Delta_\alpha\) has a strictly positive first
eigenvalue, the torsion function is
\(\upsilon_\alpha=(-\Delta_\alpha)^{-1}1\), and torsional rigidity is obtained
by integrating this positive function over all edges. An analogous
variational characterization holds in this setting; see
Theorem~\ref{thm:Polya-signed}. For compact metric graphs with Dirichlet conditions at a nonempty
set of vertices and standard Kirchhoff conditions elsewhere,
Mugnolo and Pl\"umer \cite{MugnoloPluemer2023} studied the dependence
of torsional rigidity on graph geometry and its relation to
ground-state energy. The corresponding optimization problems
on graphs ask for the largest and smallest possible torsional
rigidity when the total length and constraints on the vertex
strengths are prescribed.
	For non-negative strengths with total \(I_+>0\), \"Ozcan and T\"aufer
	\cite[Theorem~VIII.2]{OzcanTaufer2026} proved the sharp bound
	\[
	\Tor(\G,\alpha)
	\le \Tor(\mathcal J_L;I_+,0)
	=\frac{L^3}{3}+\frac{L^2}{I_+},
	\]
	where \(L\) is the total length. The maximizing configuration is an
	interval of length \(L\), with all strength concentrated at one endpoint
	and a Neumann condition at the other. 
	
	Negative strengths introduce a further issue: a positive torsion
	function need not exist. Although positive net strength is necessary,
	it is not sufficient, and the placement of the strengths matters.
	We use the discrete framework of \cite[Section~V]{OzcanTaufer2026}
	to describe positivity and torsional rigidity through vertex values.
	A path with three vertices and two edges illustrates how moving a
	negative strength changes the sign of the first eigenvalue while the
	total length and the strength values remain fixed.
	
	Our main results determine the supremum and infimum of torsional rigidity
	under fixed aggregate data. The first extends the Saint--Venant comparison
	to signed strengths and determines exactly when a finite universal upper
	bound is possible.
	Let \(I_-\) be the sum of the magnitudes of the negative strengths, with
	\(I_+\) now the total positive strength; see \eqref{eq:Ipm}. For fixed
	\(L>0\) and \(I_+>I_-\ge0\), consider all finite compact connected
	graphs with these data and strictly positive torsion.
	Theorem~\ref{thm:universal} gives the sharp comparison
	\[
	\Tor(\G,\alpha)\le\Tor(\mathcal J_L;I_+,-I_-)
	\qquad\text{if}\qquad
	I_-<\frac{I_+}{1+LI_+}.
	\]
	Here the model interval carries all positive strength at one endpoint
	and all negative strength at the other. This strict inequality guarantees positive
	torsion for every configuration with these data, and equality in the
	comparison characterizes the model interval up to zero-strength
	subdivisions. At and beyond the threshold, positive torsion can still
	exist, but there is no finite universal upper bound: three-vertex paths
	with the same data have torsional rigidity tending to infinity.
	
	Theorem~\ref{thm:sharp-infimum} determines the sharp infimum in the
	same class. Extending the lower estimate from
	\cite[proof of Theorem~VIII.1]{OzcanTaufer2026} to signed strengths,
	we obtain the infimum \(L^2/(I_+-I_-)\) throughout the range
	\(I_+>I_-\ge0\). This value is not attained by any finite graph.
	It is approached by graphs consisting of two vertices carrying strengths
	\(I_+\) and \(-I_-\), joined by increasingly many parallel edges of equal length
	and fixed total length \(L\). The construction preserves both strength
	totals separately. Thus the lower endpoint depends only on the total
	signed strength.
	
	Sections~\ref{sec:preliminaries}--\ref{sec:vertex} establish the framework,
	positivity criterion, and torsion formulas.
	Section~\ref{sec:detailed-example} examines the path example, and
	Section~\ref{sec:universal} proves the sharp comparison and the
	unboundedness result. Section~\ref{sec:sharp-infimum} determines the
	sharp infimum and proves nonattainment.

	\section{Preliminaries}\label{sec:preliminaries}
	
	Throughout, \(\G=(\V,\E)\) is a finite compact connected metric graph
	with at least one edge:
	each edge \(\mathrm e\in\E\) is an interval
	\([0,\ell_{\mathrm e}]\), with \(\ell_{\mathrm e}>0\), and its endpoints
	are identified according to the vertex set \(\V\). We denote the length vector by
	\(\ell:=(\ell_{\mathrm e})_{\mathrm e\in\E}\).
	Integration uses Lebesgue measure on each edge, and
	\[
	L:=|\G|=\sum_{\mathrm e\in\E}\ell_{\mathrm e},
	\qquad
	\int_\G u\,dx:=\sum_{\mathrm e\in\E}\int_0^{\ell_{\mathrm e}}
	u_{\mathrm e}(x)\,dx,
	\]
	where \(u_{\mathrm e}\) denotes the restriction of \(u\) to an edge.
	All functions are real-valued. We set
	\[
	L^2(\G):=\bigoplus_{\mathrm e\in\E}L^2(0,\ell_{\mathrm e}),
	\qquad
	\|u\|_{L^2(\G)}^2:=\int_\G|u|^2\,dx.
	\]
	The Sobolev space \(H^1(\G)\) consists of continuous functions on
	\(\G\) whose edge restrictions belong to \(H^1(0,\ell_{\mathrm e})\).
	We write \(H^2(\G):=\bigoplus_{\mathrm e\in\E}H^2(0,\ell_{\mathrm e})\)
	for the edgewise \(H^2\) space, without imposing vertex conditions.
	For vertex vectors, \(\ip{g}{k}\) and \(|g|\) denote the Euclidean
	inner product and norm on \(\RR^{\V}\). For a real symmetric matrix,
	\(B\succ0\) and \(B\succeq0\) mean positive definiteness and positive
	semidefiniteness, respectively.
	
	Given a strength vector \(\alpha=(\alpha_{\mathrm v})_{\mathrm v\in\V}
	\in\RR^{\V}\), we define the quadratic form and its associated bilinear
	form by
	\begin{equation}\label{eq:form}
		\begin{aligned}
			h_\alpha[u]
			&:=\sum_{\mathrm e\in\E}\int_{\mathrm e}|u'_{\mathrm e}|^2\,dx
			+\sum_{\mathrm v\in\V}\alpha_{\mathrm v}|u(\mathrm v)|^2,
		\end{aligned}
	\end{equation}
 on the domain \[
	H^1(\G)
	:=\left\{
	u\in C(\G):
	u_{\mathrm e}\in H^1(0,\ell_{\mathrm e})
	\text{ for every }\mathrm e\in\E
	\right\}.
	\]
	The quadratic form is closed and lower bounded. Its associated self-adjoint operator
	\(H_\alpha=-\Delta_{\G,\alpha}\) acts as \(-u''_{\mathrm e}\) on each
	edge. Its domain consists of the functions in
	\(H^1(\G)\cap H^2(\G)\) satisfying
	\begin{equation}\label{eq:delta-condition}
		\sum_{\mathrm e\sim\mathrm v}\partial_\nu u_{\mathrm e}(\mathrm v)
		=\alpha_{\mathrm v}u(\mathrm v),
		\qquad \mathrm v\in\V.
	\end{equation}
	Here \(\partial_\nu\) points away from the vertex into the edge: it is
	\(u'_{\mathrm e}(0)\) at the initial endpoint and
	\(-u'_{\mathrm e}(\ell_{\mathrm e})\) at the terminal endpoint.
 At a vertex with zero
	strength this is the Kirchhoff condition.
	
	The compact embedding \(H^1(\G)\hookrightarrow L^2(\G)\) gives
	\(H_\alpha\) compact resolvent. Its eigenvalues, repeated according to
	multiplicity, are denoted by
	\(\lambda_1(\G,\alpha)\le\lambda_2(\G,\alpha)\le\cdots\), and
	\[
	\lambda_1(\G,\alpha)
	=\min_{0\ne u\in H^1(\G)}
	\frac{h_\alpha[u]}{\|u\|_{L^2(\G)}^2}.
	\]
	The first eigenvalue is simple and has a strictly positive
	eigenfunction \(\varphi\).

	\section{Positivity and the aggregate comparison}\label{sec:positive}
	
	\subsection{The vertex matrix}
	
	For \(g\in\RR^{\V}\), the weighted combinatorial Laplacian
	\(K_\ell\), with conductance \(\ell_{\mathrm e}^{-1}\) on each edge,
	acts as
	\begin{equation}\label{eq:Kell}
		\begin{aligned}
			(K_\ell g)(\mathrm v)
			&=\sum_{\mathrm e=[\mathrm v,\mathrm w]\sim\mathrm v}
			\frac{g(\mathrm v)-g(\mathrm w)}{\ell_{\mathrm e}}.
		\end{aligned}
	\end{equation}
	 Set
	\(
	M_\alpha:=K_\ell+\operatorname{diag}(\alpha).
	\)
	Define the incident-length column vector and its diagonal matrix by
	\begin{equation}\label{eq:dv}
		d_{\mathrm v}:=\sum_{\mathrm e\sim\mathrm v}\ell_{\mathrm e},
		\qquad d:=(d_{\mathrm v})_{\mathrm v\in\V}^{\mathsf T},
		\qquad D_\ell:=\operatorname{diag}(d).
	\end{equation}
	The discrete operator in \cite[Section~V]{OzcanTaufer2026} is
	\(2D_\ell^{-1}M_\alpha\). It is similar to the symmetric matrix
	\(2D_\ell^{-1/2}M_\alpha D_\ell^{-1/2}\), which is positive definite
	exactly when \(M_\alpha\succ0\). Hence the positivity criterion of
	that paper can be expressed equivalently in terms of \(M_\alpha\).
	
	\begin{lemma}\label{lem:decomposition}
		Every \(u\in H^1(\G)\) has a unique decomposition
		\(u=u_0+u_{\mathrm{lin}}\), where \(u_0\) vanishes at every vertex
		and \(u_{\mathrm{lin}}\) is affine on each edge with vertex trace
		\(g=u|_{\V}\). Moreover,
		\begin{equation}\label{eq:form-decomposition}
			h_\alpha[u]
			=\sum_{\mathrm e\in\E}\int_{\mathrm e}|u'_{0,\mathrm e}|^2\,dx
			+\ip{M_\alpha g}{g}.
		\end{equation}
	\end{lemma}
	
	\begin{proof}
		On an edge \(\mathrm e=[\mathrm v,\mathrm w]\), identified with
		\([0,\ell_{\mathrm e}]\), set
		\[
		u_{\mathrm{lin},\mathrm e}(x)
		=g(\mathrm v)+\frac{g(\mathrm w)-g(\mathrm v)}{\ell_{\mathrm e}}x,
		\qquad u_{0,\mathrm e}:=u_{\mathrm e}-u_{\mathrm{lin},\mathrm e}.
		\]
		The affine interpolant is uniquely determined by its endpoint values,
		so these definitions give the claimed decomposition, including on
		loops. Since \(u_{0,\mathrm e}\) vanishes at both endpoints and
		\(u'_{\mathrm{lin},\mathrm e}\) is constant, their derivative cross
		term vanishes. Hence
		\[
		\int_{\mathrm e}|u'_{\mathrm e}|^2\,dx
		=\int_{\mathrm e}|u'_{0,\mathrm e}|^2\,dx
		+\frac{|g(\mathrm v)-g(\mathrm w)|^2}{\ell_{\mathrm e}}.
		\]
		Summing and adding the vertex terms in \eqref{eq:form} gives
		\eqref{eq:form-decomposition}.
	\end{proof}
	
	\begin{proposition}\label{prop:metric-matrix}
		For every \(\alpha\in\RR^{\V}\),
		\begin{equation}\label{eq:metric-matrix-positive}
			\lambda_1(\G,\alpha)>0
			\quad\Longleftrightarrow\quad M_\alpha\succ0.
		\end{equation}
	\end{proposition}
	
	\begin{proof}
		If \(M_\alpha\succ0\), the two terms on the right-hand side of
		\eqref{eq:form-decomposition} are non-negative and can vanish together
		only when \(g=0\) and \(u_0=0\). Thus \(h_\alpha[u]>0\) for every
		nonzero \(u\in H^1(\G)\). Since the Rayleigh minimum is attained
		at a nonzero first eigenfunction,
		\(\lambda_1(\G,\alpha)>0\).
		Conversely, if the first eigenvalue is positive, the affine extension
		\(u_{\mathrm{lin}}\) of any nonzero vertex vector \(g\) satisfies
		\[
		\ip{M_\alpha g}{g}
		=h_\alpha[u_{\mathrm{lin}}]
		\ge\lambda_1(\G,\alpha)\|u_{\mathrm{lin}}\|_{L^2(\G)}^2>0.
		\]
		This proves \(M_\alpha\succ0\).
	\end{proof}
	
 Testing \(M_\alpha\) on the all-ones
	vertex vector \(\mathbf1:=(1)_{\mathrm v\in\V}^{\mathsf T}\) also
	shows that positive first eigenvalue requires
	\(\sum_{\mathrm v\in\V}\alpha_{\mathrm v}>0\).
	
	\subsection{The aggregate comparison}
	
	Define the positive and negative strength totals by
	\begin{equation}\label{eq:Ipm}
		I_+(\alpha):=\sum_{\substack{\mathrm v\in\V\\\alpha_{\mathrm v}>0}}
		\alpha_{\mathrm v},
		\qquad
		I_-(\alpha):=\sum_{\substack{\mathrm v\in\V\\\alpha_{\mathrm v}<0}}
		|\alpha_{\mathrm v}|.
	\end{equation}
	We write \(I_+\) and \(I_-\) when the strength vector is clear.
	Together with \(L\), these are the aggregate data, and the total
	strength is \(I_+-I_-\).
	
	For \(a>0\), let \(\mathcal J_a=[0,a]\) have endpoint strengths
	\(I_+\) and \(-I_-\). Its quadratic form is
	\begin{equation}\label{eq:model-form}
		h_{\mathcal J_a,I_+,I_-}[f]
		:=\int_0^a|f'(x)|^2\,dx+I_+|f(0)|^2-I_-|f(a)|^2,
		\qquad f\in H^1(0,a).
	\end{equation}
The comparison theorem of
	Karreskog, Kurasov, and Trygg Kupersmidt
	\cite[Theorem~1]{KKT2016} gives
	\begin{equation}\label{eq:KKT-bound}
		\lambda_1(\G,\alpha)
		\ge\lambda_1(\mathcal J_L;I_+,-I_-).
	\end{equation}
	The vertex matrix of \(\mathcal J_L\) is
	\begin{equation}\label{eq:model-matrix}
		M_{\mathcal J}
		=\begin{pmatrix}
			L^{-1}+I_+&-L^{-1}\\
			-L^{-1}&L^{-1}-I_-
		\end{pmatrix},
	\end{equation}
	with determinant
	\begin{equation}\label{eq:model-determinant}
		\det M_{\mathcal J}
		=\frac{I_+-I_--LI_+I_-}{L}.
	\end{equation}
	Since its first diagonal entry is positive, Sylvester's criterion
	yields
	\begin{equation}\label{eq:aggregate-threshold}
		M_{\mathcal J}\succ0
		\quad\Longleftrightarrow\quad
		I_-<\frac{I_+}{1+LI_+}.
	\end{equation}
	By \eqref{eq:KKT-bound} and Proposition~\ref{prop:metric-matrix}, this
	strict aggregate inequality guarantees \(M_\alpha\succ0\) for every
	graph with those aggregate data.
	
	If \(u\) is any weak solution of \(H_\alpha u=1\), testing against
the first positive eigenfunction	\(\varphi\) gives
	\[
	(1,\varphi)_{L^2(\G)}
	=\lambda_1(\G,\alpha)(u,\varphi)_{L^2(\G)}.
	\]
	If \(\lambda_1(\G,\alpha)=0\), the right-hand side vanishes,
	whereas
	\((1,\varphi)_{L^2(\G)}=\int_\G\varphi\,dx>0\)
	because \(\varphi\) is strictly positive; hence no weak solution
	can exist. Moreover, if \(u\) is strictly positive, then
	\((u,\varphi)_{L^2(\G)}>0\), and the same identity implies
	\(\lambda_1(\G,\alpha)>0\).
	A negative first eigenvalue alone does not decide solvability, since
	zero may still occur among the higher eigenvalues.

	By \cite[Theorem~V.3]{OzcanTaufer2026}, strict positivity of the
	spectrum of \(2D_\ell^{-1}M_\alpha\) is equivalent to existence and
	uniqueness of a strictly positive solution of
	\(H_\alpha\upsilon_\alpha=1\). The similarity above and
	Proposition~\ref{prop:metric-matrix} therefore identify this regime
	exactly with \(M_\alpha\succ0\) and
	\(\lambda_1(\G,\alpha)>0\). We use the following definition
	throughout this regime.
	
	\begin{definition}\label{def:torsion}
		Let \(\alpha\in\RR^{\V}\) satisfy \(M_\alpha\succ0\).
		The \emph{torsion function} associated with \((\G,\alpha)\) is the
		unique function \(\upsilon_\alpha\in H^1(\G)\) such that
		\[
	H_\alpha \upsilon_\alpha=1.
		\]
	The \emph{torsional rigidity} is
		\begin{equation}\label{eq:T-definition}
			\Tor(\G,\alpha):=\int_\G\upsilon_\alpha(x)\,dx.
		\end{equation}
	\end{definition}

	In particular, this definition applies whenever the strict inequality
	in \eqref{eq:aggregate-threshold} holds.
	
	\section{Torsional rigidity and the vertex formula}\label{sec:vertex}
	
	We first record the variational characterization in the positive-torsion
	regime of Definition~\ref{def:torsion}.
	
	\begin{theorem}
		\label{thm:Polya-signed}
		Assume that \(\lambda_1(\G,\alpha)>0\), or equivalently,
		\(M_\alpha\succ0\). Then
		\begin{equation}\label{eq:Polya-signed}
			\Tor(\G,\alpha)
			=\max_{0\ne u\in H^1(\G)}
			\frac{\left(\int_\G |u(x)|\,dx\right)^2}{h_\alpha[u]}.
		\end{equation}
		The maximizers are precisely \(u=c\,\upsilon_\alpha\),
		where \(c\in\RR\setminus\{0\}\).
	\end{theorem}
	
	\begin{proof}
	The characterization and its equality case follow from the argument
	of \cite[Theorem~III.1]{OzcanTaufer2026}, which also applies to
	signed strengths under the assumption \(\lambda_1(\G,\alpha)>0\).
	\end{proof}
	
	For a positive torsion function, define its vertex column vector by
	\begin{equation}\label{eq:torsion-vertex-vector}
		g_\alpha:=
		\bigl(\upsilon_\alpha(\mathrm v)\bigr)_{\mathrm v\in\V}^{\mathsf T}.
	\end{equation}
	The incident-length vector \(d\) is defined in \eqref{eq:dv}.
	Solving the edge equations and imposing the vertex conditions yields
	the reduction from \cite[Proposition~V.2]{OzcanTaufer2026}, together
	with the following formula.
	
	\begin{proposition}
		\label{prop:T-matrix}
		If \(M_\alpha\succ0\), then
		\begin{equation}\label{eq:Mg-d}
			M_\alpha g_\alpha=\frac12d,
		\end{equation}
		and
		\begin{equation}\label{eq:T-matrix}
			\Tor(\G,\alpha)
			=\frac1{12}\sum_{\mathrm e\in\E}\ell_{\mathrm e}^3
			+\frac14\ip{d}{M_\alpha^{-1}d}.
		\end{equation}
	\end{proposition}
	
	\begin{proof}
		Let \(\mathrm e=[\mathrm v,\mathrm w]\) have length
		\(\ell_{\mathrm e}\), with coordinate \(s\in[0,\ell_{\mathrm e}]\)
		starting at \(\mathrm v\). 
		The equation \(-\upsilon_{\alpha,\mathrm e}''=1\) and the endpoint
		values give
		\begin{equation}\label{eq:torsion-edge-profile}
			\upsilon_{\alpha,\mathrm e}(s)
			=\frac{s(\ell_{\mathrm e}-s)}2
			+\left(1-\frac{s}{\ell_{\mathrm e}}\right)g_\alpha(\mathrm v)
			+\frac{s}{\ell_{\mathrm e}}g_\alpha(\mathrm w).
		\end{equation}
		The derivatives directed away from the endpoints are
		\[
		\partial_\nu\upsilon_{\alpha,\mathrm e}(\mathrm v)
		=\frac{\ell_{\mathrm e}}2
		+\frac{g_\alpha(\mathrm w)-g_\alpha(\mathrm v)}{\ell_{\mathrm e}},
		\qquad
		\partial_\nu\upsilon_{\alpha,\mathrm e}(\mathrm w)
		=\frac{\ell_{\mathrm e}}2
		+\frac{g_\alpha(\mathrm v)-g_\alpha(\mathrm w)}{\ell_{\mathrm e}}.
		\]
		Summing at \(\mathrm v\), the vertex condition
		\eqref{eq:delta-condition} becomes
		\(d_{\mathrm v}/2-(K_\ell g_\alpha)(\mathrm v)
		=\alpha_{\mathrm v}g_\alpha(\mathrm v)\).
		This is \eqref{eq:Mg-d}, and hence
		\(g_\alpha=\tfrac12M_\alpha^{-1}d\).
		
		Integrating \eqref{eq:torsion-edge-profile} gives
		\[
		\int_{\mathrm e}\upsilon_{\alpha,\mathrm e}(s)\,ds
		=\frac{\ell_{\mathrm e}^3}{12}
		+\frac{\ell_{\mathrm e}}2
		\bigl(g_\alpha(\mathrm v)+g_\alpha(\mathrm w)\bigr).
		\]
		Each endpoint incidence is counted once when these expressions are
		summed, so
		\begin{equation}\label{eq:T-vertex-trace}
			\Tor(\G,\alpha)
			=\frac1{12}\sum_{\mathrm e\in\E}\ell_{\mathrm e}^3
			+\frac12\ip{d}{g_\alpha}.
		\end{equation}
		Substitution of the vertex vector obtained above proves
		\eqref{eq:T-matrix}.
	\end{proof}
	
	\section{A detailed example}\label{sec:detailed-example}
	
	This section examines what the strength values alone can tell us about
	the existence of a strictly positive torsion function, and how their
	placement affects this question. We use a path with three vertices and
	two edges, on which the total length and the strength values can be
	kept fixed while the negative strength changes position. By determining
	the spectral sign and solving the torsion equation explicitly, we
	clarify the scope of the aggregate comparison and distinguish
	solvability from strict positivity.
	
	\subsection{The graph and its vertex conditions}
	
	Fix \(L>0\), \(I_+>0\), and \(I_-\ge0\). For \(0<x<L\), let
	\(
	\G_x=(\V_x,\E_x),
	\V_x=\{\mathrm v_0,\mathrm v_x,\mathrm v_L\},
	\E_x=\{\mathrm e_1,\mathrm e_2\}.
	\)
	The edge \(\mathrm e_1\) joins \(\mathrm v_0\) to \(\mathrm v_x\)
	and has length \(x\), while \(\mathrm e_2\) joins
	\(\mathrm v_x\) to \(\mathrm v_L\) and has length \(L-x\).
	In the vertex order \((\mathrm v_0,\mathrm v_x,\mathrm v_L)\), assign
	\(
	\alpha_x=(I_+,-I_-,0)^{\mathsf T}.
	\)
	
	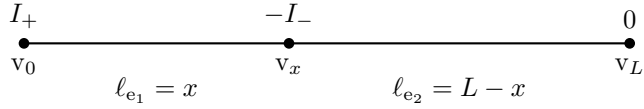
\begin{figure}[htbp]
		\centering
		\begin{tikzpicture}[
			vertex/.style={circle,fill=black,
				inner sep=0pt,minimum size=4pt}
			]
			\coordinate (v0) at (0,0);
			\coordinate (vx) at (3.5,0);
			\coordinate (vL) at (8,0);
			\draw[line width=0.8pt] (v0) --
			node[below=9pt] {$\ell_{\mathrm e_1}=x$} (vx);
			\draw[line width=0.8pt] (vx) --
			node[below=9pt] {$\ell_{\mathrm e_2}=L-x$} (vL);
			\node[vertex,label=above:{$I_+$},
			label=below:{$\mathrm v_0$}] at (v0) {};
			\node[vertex,label=above:{$-I_-$},
			label=below:{$\mathrm v_x$}] at (vx) {};
			\node[vertex,label=above:{$0$},
			label=below:{$\mathrm v_L$}] at (vL) {};
		\end{tikzpicture}
		\caption{The metric path graph \(\G_x\)}
		\label{fig:aggregate-path}
	\end{figure}
	
	Varying \(x\) preserves the total length \(|\G_x|=L\), the
	combinatorial graph, and the entire ordered strength vector. In
	particular, its aggregate strengths are always \(I_+\) and \(I_-\).
	We use a global coordinate \(s\in[0,L]\), identifying the two edges
	with \([0,x]\) and \([x,L]\). Thus \(H^1(\G_x)\) is identified
	with \(H^1(0,L)\), while the vertex \(\mathrm v_x\) retains its
	specified strength. The vertex conditions for \(H_{\alpha_x}\) are
	\begin{equation}\label{eq:detailed-path-conditions}
		\begin{gathered}
			u'(0)=I_+u(0),\qquad u'(L)=0,\\
			u(x-)=u(x+),\qquad
			u'(x+)-u'(x-)=-I_-u(x).
		\end{gathered}
	\end{equation}
	The vertex matrix and incident-length vector in this order are
	\begin{equation}\label{eq:path-matrix}
		M_{\alpha_x}=
		\begin{pmatrix}
			I_++x^{-1}&-x^{-1}&0\\
			-x^{-1}&x^{-1}+(L-x)^{-1}-I_-&-(L-x)^{-1}\\
			0&-(L-x)^{-1}&(L-x)^{-1}
		\end{pmatrix},
	\end{equation}
	\begin{equation}\label{eq:path-d}
		d_x=(x,L,L-x)^{\mathsf T}.
	\end{equation}
	Indeed, the two reciprocal edge lengths contribute the corresponding
	diagonal and off-diagonal entries in \eqref{eq:path-matrix}, and the
	strengths contribute its remaining diagonal terms. At the middle
	vertex, the two incident lengths sum to \(L\), giving the middle
	entry of \eqref{eq:path-d}.
	
	\subsection{The sign of the first eigenvalue}
	
	For a vertex trace \(g=(a,b,c)^{\mathsf T}\), the matrix energy is
	\[
	\ip{M_{\alpha_x}g}{g}
	=\frac{(a-b)^2}{x}+\frac{(b-c)^2}{L-x}+I_+a^2-I_-b^2.
	\]
	Set
	\begin{equation}\label{eq:detailed-path-D}
		D_x:=I_+-I_--xI_+I_-.
	\end{equation}
	Completing the square gives
	\begin{equation}\label{eq:detailed-path-squares}
		\ip{M_{\alpha_x}g}{g}
		=\left(I_++\frac1x\right)
		\left(a-\frac{b}{1+xI_+}\right)^2
		+\frac{D_x}{1+xI_+}b^2
		+\frac{(b-c)^2}{L-x}.
	\end{equation}
	All denominators are positive. If \(D_x>0\), this expression is
	positive for every nonzero \(g\), and
	Proposition~\ref{prop:metric-matrix} therefore gives
	\(\lambda_1(\G_x,\alpha_x)>0\).
	
	For the other two signs, consider the positive edgewise affine function
	\begin{equation}\label{eq:detailed-path-kernel}
		\varphi_{0,x}(s)=1+I_+\min\{s,x\}.
	\end{equation}
	Its vertex trace is \((1,1+xI_+,1+xI_+)^{\mathsf T}\), so the
	first and last squares in \eqref{eq:detailed-path-squares} vanish.
	By the affine--Dirichlet decomposition \eqref{eq:form-decomposition},
	\(
	h_{\alpha_x}[\varphi_{0,x}]=D_x(1+xI_+).
	\)
	If \(D_x=0\), the same decomposition and
	\eqref{eq:detailed-path-squares} show that the full form is
	nonnegative, and \(\varphi_{0,x}\) has zero Rayleigh quotient.
	Hence \(\lambda_1(\G_x,\alpha_x)=0\). If \(D_x<0\), this
	test function has negative Rayleigh quotient. Thus
	\begin{equation}\label{eq:aggregate-path-sign}
		\operatorname{sgn}\lambda_1(\G_x,\alpha_x)
		=\operatorname{sgn}D_x.
	\end{equation}
	The length \(L-x\) does not appear in this criterion because the
	second edge can carry a constant function: its derivative energy
	then vanishes, and its terminal vertex has strength zero.
	
	\subsection{Solving the torsion equation}
	
	We solve \(H_{\alpha_x}u=1\) without initially imposing positivity.
	This means solving \(-u''=1\) on each edge together with
	\eqref{eq:detailed-path-conditions}. On the first edge, write
	\(u(s)=-s^2/2+As+C\); the condition at zero gives \(A=I_+C\).
	The condition at \(L\) gives \(u(s)=-s^2/2+Ls+B\) on the second
	edge, and continuity at \(x\) gives \(B=(1+xI_+)C-Lx\).
	Finally, the derivative jump becomes
	\[
	L-I_+C=-I_-\left(-\frac{x^2}{2}+(1+xI_+)C\right),
	\]
	or equivalently,
	\begin{equation}\label{eq:detailed-path-coefficient}
		D_xC=L-\frac{I_-x^2}{2}.
	\end{equation}
	For \(D_x\ne0\), the unique solution is therefore
	\begin{equation}\label{eq:detailed-path-solution}
		\begin{gathered}
			c_x:=\frac{L-I_-x^2/2}{D_x},\\[1mm]
			u_x(s)=
			\begin{cases}
				-\dfrac{s^2}{2}+I_+c_xs+c_x,
				&0\le s\le x,\\[2mm]
				-\dfrac{s^2}{2}+Ls+(1+xI_+)c_x-Lx,
				&x\le s\le L.
			\end{cases}
		\end{gathered}
	\end{equation}
	Uniqueness follows by applying the same calculation to the homogeneous
	equation: its constant \(C\) satisfies \(D_xC=0\), which forces
	all its coefficients to vanish. This argument applies when
	\(D_x<0\).
	
	If \(D_x>0\), then \(xI_-<1\), so the numerator of \(c_x\) is
	greater than \(L-x/2>0\). Moreover,
	\[
	I_+c_x-x
	=\frac{I_+(L-x)+I_-x+I_+I_-x^2/2}{D_x}>0.
	\]
	Hence \(u_x'(s)=I_+c_x-s>0\) on the first edge and
	\(u_x'(s)=L-s>0\) in the interior of the second edge. Its minimum
	is \(c_x>0\), and \(u_x=\upsilon_{\alpha_x}\).
	
	If \(D_x<0\), its value at the negative vertex is
	\begin{equation}\label{eq:detailed-path-negative-value}
		u_x(x)=\frac{L(1+xI_+)-I_+x^2/2}{D_x}<0.
	\end{equation}
	The numerator equals \(L+I_+x(L-x/2)>0\). The solution is therefore
	not strictly positive; its sign elsewhere can depend on the strengths.
	
	If \(D_x=0\), then \(I_-=I_+/(1+xI_+)\), so
	\(L-I_-x^2/2>L-x/2>0\). Equation
	\eqref{eq:detailed-path-coefficient} is therefore inconsistent, and
	no solution exists. The positive zero eigenfunction is precisely
	\(\varphi_{0,x}\) from \eqref{eq:detailed-path-kernel}: its derivative
	jump is \(-I_+=-I_-\varphi_{0,x}(x)\), and both endpoint conditions
	hold. This gives a concrete instance of the zero-ground-state
	obstruction discussed in the preliminaries.
	
	\subsection{The four aggregate regimes}
	
	We now compare the exact sign criterion \eqref{eq:aggregate-path-sign}
	with the aggregate threshold \eqref{eq:aggregate-threshold}.
	
	\paragraph{Case 1: Non-positive total strength.}
	
	Suppose that \(I_+\le I_-\). Then \(I_->0\) and \(D_x<0\)
	for every \(0<x<L\). Every member of the path family therefore
	has a negative first eigenvalue and a unique solution which is
	negative at \(\mathrm v_x\). More generally, the necessary condition
	\(I_+>I_-\), established after Proposition~\ref{prop:metric-matrix},
	excludes a strictly positive torsion function on every graph with
	these aggregate data.
	
	For example, take \(L=1\), \(I_+=I_-=2\), and \(x=1/2\).
	Substitution in \eqref{eq:detailed-path-solution} gives
	\[
	D_{1/2}=-2,\qquad c_{1/2}=-\frac38,\qquad
	u_{1/2}(1/2)=-\frac78,\qquad u_{1/2}(1)=-\frac34.
	\]
	The derivative on the first edge is \(-s-3/4<0\), while on the
	second edge it is \(1-s>0\). Thus the solution decreases to the
	middle vertex and then increases, and its global maximum is
	\(u_{1/2}(0)=-3/8<0\). At the middle vertex the derivative jump is
	\(7/4=-2u_{1/2}(1/2)\), as required. The equation is solvable,
	although its unique solution is strictly negative throughout.
	
	\paragraph{Case 2: Below the aggregate threshold.}
	
	Suppose that \(I_-<I_+/(1+LI_+)\). For every \(0<x<L\),
	\[
	D_x=(I_+-I_--LI_+I_-)+(L-x)I_+I_->0.
	\]
	Hence every placement in this family has positive torsion.
	The aggregate comparison already guarantees the same conclusion
	for every connected compact graph with these data. The following
	elementary estimate also explains what happens at the threshold.
	
	For \(u\in H^1(\G)\), put \(m=\min_\G|u|\) and
	\(M=\max_\G|u|\). A simple path joining points attaining these
	extrema has length at most \(L\), so Cauchy--Schwarz gives
	\(\int_\G|u'|^2\ge(M-m)^2/L\); when \(m=M\), this is immediate.
	Bounding the vertex terms by \(I_+m^2-I_-M^2\) and completing the
	square yields
	\begin{align}
		h_\alpha[u]
		&\ge\frac{(M-m)^2}{L}+I_+m^2-I_-M^2\notag\\
		&=\frac{\bigl((1+LI_+)m-M\bigr)^2}{L(1+LI_+)}
		+\left(\frac{I_+}{1+LI_+}-I_-\right)M^2.
		\label{eq:detailed-aggregate-energy}
	\end{align}
	Since \(\|u\|_{L^2(\G)}^2\le LM^2\), the Rayleigh principle gives
	\[
	\lambda_1(\G,\alpha)
	\ge\frac1L\left(\frac{I_+}{1+LI_+}-I_-\right)>0.
	\]
	In particular, a zero first eigenvalue is impossible in this regime.
	
	For a numerical example, choose \(L=1\), \(I_+=2\),
	\(I_-=1/2\), and \(x=1/2\). The threshold is \(2/3\), and
	\eqref{eq:detailed-path-solution} gives
	\[
	D_{1/2}=1,\qquad c_{1/2}=\frac{15}{16},\qquad
	\upsilon_{\alpha_{1/2}}(1/2)=\frac74,\qquad
	\upsilon_{\alpha_{1/2}}(1)=\frac{15}{8}.
	\]
	Its derivative on the first edge is \(15/8-s>0\), and on the
	second it is \(1-s>0\) except at the terminal endpoint. Its minimum
	is therefore \(15/16\). The derivative jump is
	\(-7/8=-(1/2)\upsilon_{\alpha_{1/2}}(1/2)\).
	
	\paragraph{Case 3: At the aggregate threshold.}
	
	Suppose that \(I_-=I_+/(1+LI_+)\). Now
	\eqref{eq:detailed-aggregate-energy} shows that the full form is
	nonnegative on every connected compact graph with these aggregate
	data, so every such graph has \(\lambda_1\ge0\).
	For the interior negative placement \(\alpha_x\), however,
	\(
	D_x=(L-x)I_+I_->0, (0<x<L),
	\)
	and the torsion function remains strictly positive.
	
	To obtain a zero first eigenvalue on the same three-vertex, two-edge
	graph, exchange the negative and zero strengths, setting
	\(
	\widehat\alpha_x=(I_+,0,-I_-)^{\mathsf T}.
	\)
	The interior vertex now has matching one-sided derivatives, and the
	endpoint condition at \(L\) is \(u'(L)=I_-u(L)\). The function
	\(
	\widehat\varphi_0(s)=1+I_+s
	\)
	satisfies these conditions, because
	\(I_-=I_+/(1+LI_+)\), and is a positive zero eigenfunction.
	Nonnegativity of the form implies
	\(\lambda_1(\G_x,\widehat\alpha_x)=0\).
	Testing a hypothetical torsion solution against it would give
	\[
	0=h_{\widehat\alpha_x}[u,\widehat\varphi_0]
	=\int_0^L\widehat\varphi_0(s)\,ds
	=L+\frac{I_+L^2}{2}>0.
	\]
	Thus this endpoint placement admits no solution.
	
	For instance, fix \(L=1\), \(x=1/2\), \(I_+=2\), and
	\(I_-=2/3\). With the interior placement, formula
	\eqref{eq:detailed-path-solution} gives
	\[
	D_{1/2}=\frac23,\qquad c_{1/2}=\frac{11}{8},\qquad
	\upsilon_{\alpha_{1/2}}(1/2)=\frac{21}{8},\qquad
	\upsilon_{\alpha_{1/2}}(1)=\frac{11}{4}.
	\]
	The first-edge derivative is \(11/4-s>0\), and the second-edge
	derivative is \(1-s\). The minimum is \(11/8>0\), and the jump
	\(-7/4\) equals \(-(2/3)(21/8)\). With the endpoint placement,
	the positive zero eigenfunction is \(1+2s\), whose integral is
	\(2\), and the torsion equation is unsolvable. These two assignments
	have the same individual strengths and the same metric graph;
	only their vertex locations differ.
	
	\paragraph{Case 4: Above the aggregate threshold with positive total strength.}
	
	Finally, suppose that \(I_+/(1+LI_+)<I_-<I_+\), and define
	\begin{equation}\label{eq:path-critical-position}
		x_*:=\frac{I_+-I_-}{I_+I_-}.
	\end{equation}
	The assumptions give \(0<x_*<L\), and
	\(D_x=I_+I_-(x_*-x)\). By \eqref{eq:aggregate-path-sign} and
	the preceding solution calculation,
	\[
	\begin{array}{c|c|l}
		\text{Position}&\lambda_1(\G_x,\alpha_x)&\text{Torsion equation}\\
		\hline
		0<x<x_*&>0&\text{unique strictly positive solution}\\
		x=x_*&=0&\text{no solution}\\
		x_*<x<L&<0&\text{unique solution, negative at }\mathrm v_x
	\end{array}
	\]
	
	Take \(L=1\), \(I_+=2\), and \(I_-=1\). Then \(x_*=1/2\).
	For \(x=1/4\), substitution in
	\eqref{eq:detailed-path-solution} gives \(D_x=1/2\) and
	\(c_x=31/16\). The solution increases on both edges, with
	\[
	u_{1/4}(0)=\frac{31}{16},\qquad
	u_{1/4}(1/4)=\frac{23}{8},\qquad
	u_{1/4}(1)=\frac{101}{32}.
	\]
	It is the strictly positive torsion function.
	
	For \(x=1/2\), the zero eigenfunction from
	\eqref{eq:detailed-path-kernel} is
	\(\varphi_{0,1/2}(s)=1+2\min\{s,1/2\}\), with integral
	\(7/4\). Testing the torsion equation against this function would
	give \(0=7/4\), so there is no solution.
	
	For \(x=3/4\), formula \eqref{eq:detailed-path-solution} gives
	\(D_x=-1/2\), \(c_x=-23/16\), and
	\[
	u_{3/4}(0)=-\frac{23}{16},\qquad
	u_{3/4}(3/4)=-\frac{31}{8},\qquad
	u_{3/4}(1)=-\frac{123}{32}.
	\]
	The derivative on the first edge is \(-s-23/8<0\), and on the
	second it is \(1-s>0\). Thus the global maximum is
	\(-23/16<0\), and the unique solution is strictly negative.
	All three spectral signs have therefore occurred with the same total
	length and the same ordered strength vector.
	
	The four cases show that positive total strength is necessary for
	positive torsion, but does not guarantee it. Below the aggregate
	threshold, every configuration admits a strictly positive torsion
	function. At the threshold, the placement of the same strengths can
	produce either positive torsion or a zero first eigenvalue. Above the
	threshold, while the total strength remains positive, all three
	spectral signs are possible with the same strength values and total
	length. Thus the aggregate condition provides a uniform guarantee,
	but is not necessary for positivity on an individual graph. The examples
	also distinguish positivity from solvability: a zero first eigenvalue
	excludes every solution of the torsion equation, whereas a negative
	first eigenvalue may still allow a solution that is not strictly positive.
	
	\section{Sharp universal bounds from aggregate strengths}
	\label{sec:universal}
	
	We now determine when the aggregate data yield a finite upper bound
	over all graphs with strictly positive torsion. We first compute the
	value of the comparison interval introduced in \eqref{eq:model-form}.
	
	\begin{example}\label{ex:model-interval}
		Suppose that the strict inequality in \eqref{eq:aggregate-threshold}
		holds. The torsion function on \(\mathcal J_L\) satisfies
		\[
		-\upsilon_{\mathcal J}''=1,\qquad
		\upsilon_{\mathcal J}'(0)=I_+\upsilon_{\mathcal J}(0),\qquad
		\upsilon_{\mathcal J}'(L)=I_-\upsilon_{\mathcal J}(L).
		\]
		Thus
		\[
		\upsilon_{\mathcal J}(x)=-\frac{x^2}{2}+I_+c_0x+c_0,
		\qquad
		c_0=\frac{L\left(1-\frac{LI_-}{2}\right)}
		{I_+-I_--LI_+I_-}.
		\]
		Since \(\upsilon_{\mathcal J}>0\), the endpoint condition gives
		\[
		\upsilon_{\mathcal J}'(x)
		=I_-\upsilon_{\mathcal J}(L)+L-x>0
		\qquad(0\le x<L).
		\]
		Thus its minimum and maximum occur at the two endpoints, including
		when \(I_-=0\).
		Integrating the torsion function gives
		\begin{equation}\label{eq:model-T-explicit}
			\Tor(\mathcal J_L;I_+,-I_-)
			=-\frac{L^3}{6}
			+\frac{L^2
				\left(1-\frac{LI_-}{2}\right)
				\left(1+\frac{LI_+}{2}\right)}
			{I_+-I_--LI_+I_-}.
		\end{equation}
		When \(I_-=0\), this reduces to \(L^3/3+L^2/I_+\), the model
		value in the non-negative Saint--Venant comparison of
		\cite{OzcanTaufer2026}.
	\end{example}
	
	\begin{theorem}\label{thm:universal}
		Fix \(L>0\) and \(I_+>I_-\ge0\), and let
		\(\mathfrak C(L,I_+,I_-)\) denote the class of pairs \((\G,\alpha)\)
		such that \(\G\) is a finite compact connected metric graph,
		\(|\G|=L\), \(M_\alpha\succ0\), and
		\(I_+(\alpha)=I_+\), \(I_-(\alpha)=I_-\). Then
		\begin{equation}\label{eq:universal-finiteness}
			\sup_{(\G,\alpha)\in\mathfrak C(L,I_+,I_-)}
			\Tor(\G,\alpha)<\infty
			\quad\Longleftrightarrow\quad
			I_-<\frac{I_+}{1+LI_+}.
		\end{equation}
		When these equivalent conditions hold,
		\begin{equation}\label{eq:signed-SV}
			\Tor(\G,\alpha)
			\le\Tor(\mathcal J_L;I_+,-I_-)
			\qquad
			\text{for every }(\G,\alpha)\in\mathfrak C(L,I_+,I_-),
		\end{equation}
		where the model value is given by \eqref{eq:model-T-explicit}.
		The model interval attains the supremum. Up to insertion or removal
		of degree-two vertices of zero strength, equality holds only for
		an interval with all positive strength at the minimum endpoint of
		its torsion function and all negative strength at its maximum
		endpoint.
		
		If the condition in \eqref{eq:universal-finiteness} fails, there
		is a sequence of three-vertex paths
		\((\G_n,\alpha_n)\in\mathfrak C(L,I_+,I_-)\) such that
		\(\Tor(\G_n,\alpha_n)\to+\infty\).
	\end{theorem}
	
	\begin{proof}
		Assume first that the strict aggregate threshold holds.
		By \eqref{eq:aggregate-threshold} and
		Proposition~\ref{prop:metric-matrix}, the comparison interval has
		strictly positive first eigenvalue. Take
		\((\G,\alpha)\in\mathfrak C(L,I_+,I_-)\), and let
		\(\upsilon_\alpha\) be its strictly positive torsion function.
		Write
		\(
		m:=\min_\G\upsilon_\alpha,
		 M:=\max_\G\upsilon_\alpha.
		\)
		The edgewise equation makes this function nonconstant, so
		\(0<m<M\). Moving the signed vertex contributions to these
		extremal levels gives
		\begin{align}
			\sum_{\mathrm v\in\V}
			\alpha_{\mathrm v}\upsilon_\alpha(\mathrm v)^2
			&=\sum_{\substack{\mathrm v\in\V\\\alpha_{\mathrm v}>0}}
			\alpha_{\mathrm v}\upsilon_\alpha(\mathrm v)^2
			-\sum_{\substack{\mathrm v\in\V\\\alpha_{\mathrm v}<0}}
			|\alpha_{\mathrm v}|\upsilon_\alpha(\mathrm v)^2
			\notag\\
			&\ge I_+m^2-I_-M^2.
			\label{eq:vertex-rearrange}
		\end{align}
		Let \(\upsilon_\alpha^*\) be the nondecreasing equimeasurable
		rearrangement on \([0,L]\), and define
		\[
		\eta(t):=|\{x\in\G:\upsilon_\alpha(x)<t\}|,
		\qquad
		n(t):=\#\{x\in\G:\upsilon_\alpha(x)=t\}.
		\]
		For regular values \(t\in(m,M)\), excluding the finitely many vertex and
		stationary values, connectedness gives \(n(t)\ge1\) and
		\(\eta'(t)=\sum_{\upsilon_\alpha(x)=t}|\upsilon_\alpha'(x)|^{-1}
		\in(0,\infty)\). The coarea and Cauchy--Schwarz calculation in
		\cite[proof of Theorem~VIII.2,
		equations~(VIII.2)--(VIII.4)]{OzcanTaufer2026}
		depends only on connectedness and the edgewise equation,
		so it applies to the present signed strengths. It gives
		\(\upsilon_\alpha^*\in H^1(0,L)\),
		\begin{equation}\label{eq:L1-rearrangement}
			\upsilon_\alpha^*(0)=m,\qquad
			\upsilon_\alpha^*(L)=M,\qquad
			\int_0^L\upsilon_\alpha^*\,dx
			=\int_\G\upsilon_\alpha\,dx,
		\end{equation}
		and
		\begin{equation}\label{eq:gradient-rearrangement}
			\int_\G|\upsilon_\alpha'|^2\,dx
			\ge\int_m^M\frac{n(t)^2}{\eta'(t)}\,dt
			\ge\int_m^M\frac{dt}{\eta'(t)}
			=\int_0^L|(\upsilon_\alpha^*)'|^2\,dx.
		\end{equation}
		Equality between the graph and rearranged derivative energies
		holds exactly when \(n(t)=1\) for almost every regular value;
		in that case the Cauchy--Schwarz step is automatically an equality.
		
		Combining the vertex and derivative estimates yields
		\begin{equation}\label{eq:form-rearrangement}
			h_\alpha[\upsilon_\alpha]
			\ge h_{\mathcal J_L,I_+,I_-}[\upsilon_\alpha^*]>0.
		\end{equation}
		The weak torsion identity gives
		\(h_\alpha[\upsilon_\alpha]=\Tor(\G,\alpha)\).
		Hence \eqref{eq:L1-rearrangement} and the variational principle
		\eqref{eq:Polya-signed} give
		\begin{align*}
			\Tor(\G,\alpha)
			=\frac{\left(\int_\G\upsilon_\alpha\,dx\right)^2}
			{h_\alpha[\upsilon_\alpha]}
			\le
			\frac{\left(\int_0^L\upsilon_\alpha^*\,dx\right)^2}
			{h_{\mathcal J_L,I_+,I_-}[\upsilon_\alpha^*]}
			\le\Tor(\mathcal J_L;I_+,-I_-).
		\end{align*}
		The model interval belongs to the class and attains this bound.
		
	Suppose now that equality holds in \eqref{eq:signed-SV}.
	Equality throughout the quotient comparison above implies
	\(
	h_\alpha[\upsilon_\alpha]
	=h_{\mathcal J_L,I_+,I_-}[\upsilon_\alpha^*].
	\)
	Since the differences in the derivative and vertex estimates are
	both nonnegative, equality holds separately in
	\eqref{eq:gradient-rearrangement} and \eqref{eq:vertex-rearrange}.
	In particular, \(n(t)=1\) for almost every regular value
	\(t\in(m,M)\).
	We claim that \(\upsilon_\alpha\) is injective on \(\G\).
	Otherwise, choose distinct points \(x,y\in\G\) such that
	\(\upsilon_\alpha(x)=\upsilon_\alpha(y)\), and join them by a
	simple path. Since \(-\upsilon_{\alpha,\mathrm e}''=1\) on each
	edge, the restriction of \(\upsilon_\alpha\) to this path cannot
	be constant. Hence there is an interior point \(z\) of the path
	with \(\upsilon_\alpha(z)\ne\upsilon_\alpha(x)\).
	By continuity, every value strictly between
	\(\upsilon_\alpha(x)\) and \(\upsilon_\alpha(z)\) is attained
	on both subpaths joining \(x\) to \(z\) and \(z\) to \(y\).
	This gives two distinct preimages for every regular value in
	a nonempty interval, contradicting \(n(t)=1\) almost everywhere.
	
	Since \(\G\) is compact and connected, the continuous injective
	map \(\upsilon_\alpha\) is a homeomorphism from \(\G\) onto
	\([m,M]\). Thus \(\G\) is a path of total length \(L\), possibly
	subdivided, and its torsion function is strictly monotone along
	the path. Its minimum and maximum are therefore attained uniquely
	at the two endpoints.
	Finally, equality in \eqref{eq:vertex-rearrange} is equivalent to
	\[
	\sum_{\substack{\mathrm v\in\V\\\alpha_{\mathrm v}>0}}
	\alpha_{\mathrm v}
	\bigl(\upsilon_\alpha(\mathrm v)^2-m^2\bigr)
	+
	\sum_{\substack{\mathrm v\in\V\\\alpha_{\mathrm v}<0}}
	|\alpha_{\mathrm v}|
	\bigl(M^2-\upsilon_\alpha(\mathrm v)^2\bigr)
	=0.
	\]
	Every summand is nonnegative. Consequently, every vertex with
	positive strength is a minimum point, and every vertex with
	negative strength is a maximum point. The endpoint strengths
	are therefore \(I_+\) and \(-I_-\), while all interior vertices
	have zero strength. Removing these degree-two vertices gives
	precisely the model interval. Conversely, inserting degree-two
	vertices of zero strength into the model interval preserves
	equality.
	
		Finally, suppose that the strict threshold fails. Since
		\(I_+>I_-\), necessarily \(I_->0\). Use the critical position
		\(x_*\) defined by \eqref{eq:path-critical-position}; in the
		present range \(0<x_*\le L\), with \(x_*=L\) at the aggregate
		threshold. For \(0<x<x_*\), take the path and strength assignment
		\((\G_x,\alpha_x)\) of Section~\ref{sec:detailed-example}.
		Equation~\eqref{eq:detailed-path-D} gives \(D_x>0\), so
		\eqref{eq:aggregate-path-sign} and
		Proposition~\ref{prop:metric-matrix} show that
		\((\G_x,\alpha_x)\in\mathfrak C(L,I_+,I_-)\).
		
		By \eqref{eq:detailed-path-solution}, its torsion vertex value at
		\(\mathrm v_0\) is \(g_{\alpha_x}(\mathrm v_0)=c_x\).
		As \(x\uparrow x_*\), its denominator \(D_x\) tends to zero
		from above. Moreover, \(I_-x_*=(I_+-I_-)/I_+<1\), and hence
		\[
		L-\frac{I_-x_*^2}{2}>L-\frac{x_*}{2}\ge\frac L2>0.
		\]
		Thus \(c_x\to+\infty\). All torsion vertex values are positive,
		so \eqref{eq:path-d} and \eqref{eq:T-vertex-trace} give
		\(
		\Tor(\G_x,\alpha_x)\ge\frac{x}{2}c_x.
		\)
		Choose
		\[
		x_n:=x_*\left(1-\frac1{n+1}\right),
		\qquad (\G_n,\alpha_n):=(\G_{x_n},\alpha_{x_n}),
		\qquad n\ge1.
		\]
		Each path has two edges of strictly positive length, also when
		\(x_*=L\). Since \(x_n\to x_*\), the preceding lower bound
		shows that \(\Tor(\G_n,\alpha_n)\to+\infty\).
		This proves the remaining implication in
		\eqref{eq:universal-finiteness}.
	\end{proof}
	\section{The sharp infimum for fixed aggregate strengths}
	\label{sec:sharp-infimum}
	
	Theorem~\ref{thm:universal} determines the supremum of torsional
	rigidity over \(\mathfrak C(L,I_+,I_-)\). We now determine the
	corresponding infimum. The argument extends the lower estimate in
	\cite[proof of Theorem~VIII.1]{OzcanTaufer2026} from non-negative
	strengths to the regime \(M_\alpha\succ0\).
	
	\begin{theorem}\label{thm:sharp-infimum}
		Fix \(L>0\) and \(I_+>I_-\ge0\), and let
		\(\mathfrak C(L,I_+,I_-)\) be the class defined in
		Theorem~\ref{thm:universal}. Every
		\((\G,\alpha)\in\mathfrak C(L,I_+,I_-)\) satisfies
		\begin{equation}\label{eq:aggregate-lower-bound}
			\Tor(\G,\alpha)
			\ge \frac1{12}\sum_{\mathrm e\in\E}\ell_{\mathrm e}^3
			+\frac{L^2}{I_+-I_-}
			>\frac{L^2}{I_+-I_-}.
		\end{equation}
		Moreover,
		\begin{equation}\label{eq:sharp-infimum}
			\inf_{(\G,\alpha)\in\mathfrak C(L,I_+,I_-)}
			\Tor(\G,\alpha)
			=\frac{L^2}{I_+-I_-}.
		\end{equation}
		The infimum is not attained.
	\end{theorem}
	
	\begin{proof}
		Let \((\G,\alpha)\in\mathfrak C(L,I_+,I_-)\). Since
		\(K_\ell\mathbf1=0\), we have
		\[
		\ip{M_\alpha\mathbf1}{\mathbf1}
		=\sum_{\mathrm v\in\V}\alpha_{\mathrm v}=I_+-I_-,
		\qquad
		\ip{d}{\mathbf1}=2L.
		\]
		The second identity counts both endpoint incidences of each edge.
		Cauchy--Schwarz in the inner product induced by \(M_\alpha\) gives
		\[
		4L^2=\ip{d}{\mathbf1}^2
		\le \ip{d}{M_\alpha^{-1}d}
		\ip{M_\alpha\mathbf1}{\mathbf1}
		=(I_+-I_-)\ip{d}{M_\alpha^{-1}d}.
		\]
		Substitution into \eqref{eq:T-matrix} proves the first inequality
		in \eqref{eq:aggregate-lower-bound}. Since every edge has strictly
		positive length and \(L>0\), the cubic sum is strictly positive.
		This proves the second inequality and excludes attainment of the
		value in \eqref{eq:sharp-infimum}.
		
		To prove sharpness, for each integer \(n\ge1\), let \(\G_n\)
		consist of two vertices \(\mathrm v_+\) and \(\mathrm v_-\)
		joined by \(n\) parallel edges, each of length \(L/n\).
		In the vertex order \((\mathrm v_+,\mathrm v_-)\), assign
		\(
		\alpha_n=(I_+,-I_-)^{\mathsf T}.
		\)
		The total length is \(L\), and the aggregate strengths are
		\(I_+\) and \(I_-\) for every \(n\). Each edge has conductance
		\(n/L\), so
		\[
		M_{\alpha_n}=
		\begin{pmatrix}
			\dfrac{n^2}{L}+I_+ & -\dfrac{n^2}{L}\\[4pt]
			-\dfrac{n^2}{L} & \dfrac{n^2}{L}-I_-
		\end{pmatrix},
		\qquad
		d_n=(L,L)^{\mathsf T}.
		\]
		Its determinant is
		\[
		\det M_{\alpha_n}
		=\frac{n^2(I_+-I_-)-LI_+I_-}{L}.
		\]
		Since the first diagonal entry is positive, \(M_{\alpha_n}\succ0\)
		precisely when
\(
			n^2(I_+-I_-)>LI_+I_-.
\)
		This holds for all sufficiently large \(n\), and hence
		\((\G_n,\alpha_n)\in\mathfrak C(L,I_+,I_-)\) for those \(n\).
		Applying \eqref{eq:T-matrix} and inverting the two-by-two matrix
		gives
		\begin{equation}\label{eq:parallel-edge-torsion}
			\Tor(\G_n,\alpha_n)
			=\frac{L^3}{12n^2}
			+\frac{L^2\bigl(4n^2+L(I_+-I_-)\bigr)}
			{4\bigl(n^2(I_+-I_-)-LI_+I_-\bigr)}.
		\end{equation}
		Therefore
		\[
		\lim_{n\to\infty}\Tor(\G_n,\alpha_n)
		=\frac{L^2}{I_+-I_-}.
		\]
	 Together with \eqref{eq:aggregate-lower-bound}, this proves
		\eqref{eq:sharp-infimum}. Moreover, for every
		$(\G,\alpha)\in\mathfrak C(L,I_+,I_-)$,
		\[
		\Tor(\G,\alpha)-\frac{L^2}{I_+-I_-}
		\ge \frac1{12}\sum_{\mathrm e\in\E}\ell_{\mathrm e}^3>0.
		\]
		Thus no pair in $\mathfrak C(L,I_+,I_-)$ attains the infimum.
	\end{proof}

\end{document}